\documentclass[]{interact}

\usepackage{epstopdf}
\usepackage[caption=false]{subfig}
\usepackage{dsfont}
\usepackage[numbers,sort&compress]{natbib}
\bibpunct[, ]{[}{]}{,}{n}{,}{,}
\renewcommand\bibfont{\fontsize{10}{12}\selectfont}
\makeatletter
\def\NAT@def@citea{\def\@citea{\NAT@separator}}
\makeatother

\theoremstyle{plain}
\newtheorem{theorem}{Theorem}[section]
\newtheorem{lemma}[theorem]{Lemma}

\newtheorem{proposition}[theorem]{Proposition}

\theoremstyle{definition}
\newtheorem{definition}[theorem]{Definition}
\newtheorem{example}[theorem]{Example}

\theoremstyle{remark}
\newtheorem{remark}{Remark}

\begin{document}


\title{A Random Billiard Map in the Circle}

\author{
\name{Túlio Vales \textsuperscript{a}\thanks{Túlio Vales. Email: tuliovf@ufmg.br}, and Sônia Pinto-de-Carvalho\textsuperscript{b}\thanks{Sônia Pinto-de-Carvalho. Email: sonia@mat.ufmg.br}}
}

\maketitle

\begin{abstract}
We consider a random billiard map, the one in which the standard specular reflection rule is replaced by a random reflection given by a Markov operator. We exhibit an invariant measure for random billiards on general tables. In the special case of a circular table we show that almost every (random) orbit is dense in the boundary as well as in the circular ring  formed between the circle boundary and the random caustic. We additionally prove Strong Knudsen’s Law for a particular case of families of absolutely continuous measures with respect to Liouville.
\end{abstract}

\begin{keywords}
Circular billiards, Random Billiards, Random Maps, Dense Orbits, Lyapunov Exponents, Invariant Measure, Knudsen's Law.
\end{keywords}

\section{Introduction}

A deterministic billiard is a system where a particle moves with constant velocity inside a compact plane region (which has a smooth boundary and is called the billiard table), reflecting elastically at the impacts with the boundary. This means that the angle of incidence is equal to the angle of reflection. On a random billiard, this last property does not hold:  the angles of reflection  are chosen  by a certain probability law.
 
In this work we are particularly interested on a random circular billiard: we perturb the deterministic billiard on a circle using the random map introduced by Renato Feres in \cite{feres} to change the angle of reflection. 

The random Feres map will model a billiard table with microscopic irregularities of the triangular shape. We consider these irregularities to be dimensionally comparable to the moving particle. 

The random Feres map is generated by the deterministic billiard on an isosceles triangle whose base angles are less than $\frac{\pi}{6}$. (Figure \ref{irregularidade irregular}.) 

\begin{figure}[htp!]
\centering
\includegraphics[scale=0.35]{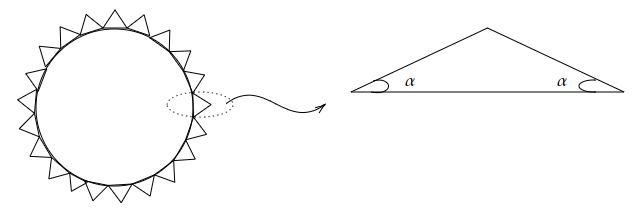}
\caption{Triangular irregularity.}
\label{irregularidade irregular}
\end{figure}

Considering segment $PQ$ with length 1 as the base of the isosceles triangle, we define the first return map to the side $PQ$ as $G(x,\theta)=(y,T_x(\theta))$ where $T_x(\theta)$ is the exit angle. See Figure \ref{colisão em um triangulo}.

\begin{figure}[htp!]
\centering
\includegraphics[scale=0.11]{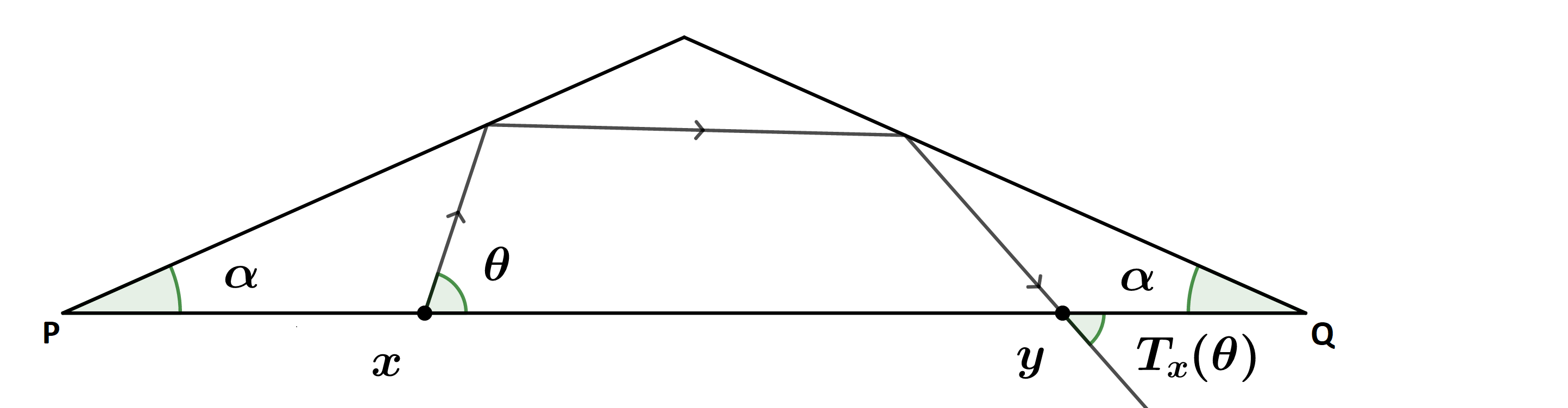}
\caption{Collision in a triangle.}
\label{colisão em um triangulo}
\end{figure}

Since the base angle of this triangle is less than $\frac{\pi}{6}$, the deterministic billiard on this triangle, starting from $PQ$, will have at most two collisions before returning to the $PQ$ side again. Therefore, there are four maps that determine the exit angle after the collision with the $PQ$ side. See Figure \ref{mapa gerado}.

\begin{figure}[htp!]
    \centering
    \includegraphics[scale=0.25]{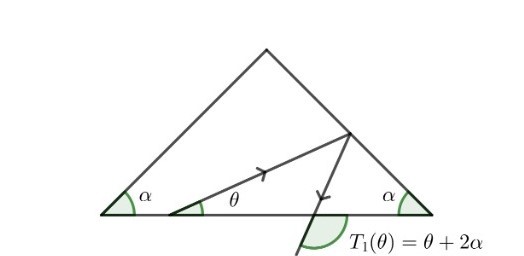}
    \includegraphics[scale=0.25]{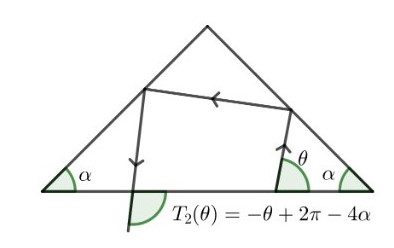}
    \includegraphics[scale=0.25]{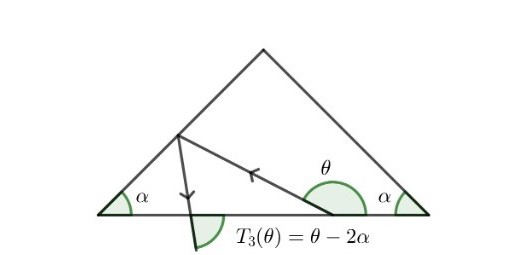}
    \includegraphics[scale=0.25]{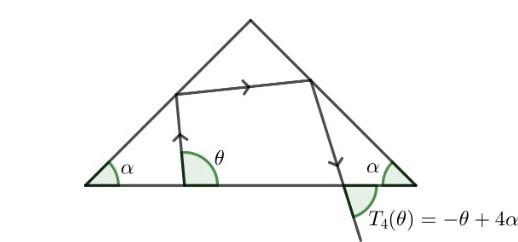}
    \caption{The random Feres map generated by triangular billiards.}
    \label{mapa gerado}
\end{figure}

The probabilities that these exit angles occur are proportional to the $PQ$ side, that is, given an initial exit angle $\theta$, the probability $P_i(\theta)$ of the exit angle after returning to the $PQ$ side is the maximum length (in terms of Lebesgue measure) that we pass from one type of collision to another among four possible. With the help of the Sine and Cosine Laws, we can explicitly determine their respective expressions. In section 3, we will describe all of them.

Kamaludin Dingle, Jeroen S. W. Lamb, and Joan-Andreu Lazaro-Cami in \cite{lamb} use the Feres's random map to study the random billiard on an infinite pipeline and proved the Strong Knudsen's Law in this case. Their strategy is to prove that the associated skew-product (as defined by Wael Bahsoun, Christopher Bose, and Anthony Quas in \cite{Quas}, which is a deterministic representation of random maps)  is an exact endomorphism.
We use the same ideas to prove that the Strong Knudsen's Law is also valid for our random circular billiard.

We also study the dynamical features of the random circular billiard, comparing its dynamics to the deterministic case. 
In fact, under certain conditions, we show the abundance of dense orbits in the boundary of the circular table.  We define and study random caustics, and show the abundance of dense orbits in the circular ring formed by the boundary of the circular table and the random caustic. We also calculate the Lyapunov exponent for the random billiard map in the circle.

 We complete this introduction with a note on the organization.  In Section 2, we make a short introduction to deterministic billiards, presenting some results already known in the literature. In Section 3, we make a brief study  of the Feres's random map and present some results to be used later. In addition, we extend a result addressed in \cite{lamb} about the Strong Knudsen's Law. In Section 4, we define an invariant measure for the random billiard map, and show that the Liouville measure is an invariant measure for that billiard. In Section 5, we make a study about Markov chains and their relation with random maps. In Section 6, we study the dynamics of the random billiard map in the circle with radius one. In fact, under certain conditions, we show the abundance of dense orbits in the boundary of the circular table. We also prove the Strong Knudsen's Law for the random billiard map in the circle for a particular family of absolutely continuous measures with respect to Liouville. Besides that, we define random caustics, and we show the abundance of dense orbits in the circular ring formed by the boundary of the circular table and the random caustic. We also calculate the Lyapunov exponent for the random billiard map in the circle. Finally, in the Appendix A, we calculate the Lyapunov exponent associated with the random billiard map the infinite pipeline.

Although we could work with random billiards in other tables, we remark that, as was proved by Bialy, in \cite{bialy1993convex}, the only $C^2$ closed and convex plane curve where the billiard map preserves the angle throughout the trajectory is the circular billiard. And since the Feres map only acts on the angles, we prefer to study the case of circular billiards. 

\section{Deterministic Billiards}

Let $\gamma$ be a plane, simple, closed, regular, smooth, oriented curve parametrized by the arc length parameter $s\in [0,L)$ and with strictly positive curvature. Let $U$ be the region enclosed by $\gamma$.

The billiard system consists in the free motion of a particle on the planar region $U$, being reflected elastically at the impacts on the boundary $\gamma$. The dynamic is then determined  by the collision point at $\gamma$ and the direction of motion, immediately after each impact. These two elements can be given by the parameter $s\in [0,L]$ that locates the point of reflection, and by the angle $\theta \in (0,\pi)$ between the tangent vector $\gamma'(s)$ and the outgoing trajectory, measured counter-clockwise.

In this context, the billiard system defines a map $F$ (called the billiard map) from the open rectangle $[0,L]\times(0,\pi)$ into itself. 
A trajectory will be a polygonal line on this planar region, i.e., a set of straight segments connecting consecutive impacts. 

If there exists a curve such that every segment of the trajectory is tangent to it, it will be called a caustic.

The set of points $\{F^n(s_0,\theta_0) : n\in \mathbb{Z}\}\subset [0,L)\times(0,\pi)$ is the orbit of the point $(s_0,\theta_0)$ in the phase space $[0,L)\times(0,\pi)$. A point $(s_0,\theta_0)$ is $n$-periodic if $F^n(s_0,\theta_0)=(s_0,\theta_0)$ or, equivalently, if the polygonal line defined by the trajectory is actually an $n$-sided polygon.

The billiard map $F$ has some well-known properties (see, for instance, \cite{birkhoff, markarian, katok, tabachnikov}): if $\gamma$ is a $C^k$-curve, then $F$ is a $C^{k-1}$-diffeomorphism, reversible with respect to the reversing symmetry $G(\varphi,\theta)=(\varphi,\pi-\theta)$ and, as $\gamma$ has strictly positive curvature, has the monotone twist property.

The billiard map $F$ also preserves the probability measure $\nu = \lambda \times \mu$, where $\lambda$ is the normalized Lebesgue measure for $[0, L)$ and $\mu(A) = \frac{1}{2} \int_{A} \sin(\theta) d\theta$ for every measurable set $A \subset (0,\pi)$. This $\lambda \times \mu$ measure is also called the Liouville measure. 

Moreover, the derivative of $F$ can be implicitly calculated and is given by 
 $$DF(s_0,\theta_0) = \frac{1}{\sin \theta_1} \left [\begin{array}{cc} l_{01} k_0- \sin \theta_0 & l_{01} \\ k_1 (l_{01} k_0- \sin \theta_0) -k_0 \sin \theta_1 & l_{01} k_1- \sin \theta_1 \end{array} \right],$$
where $(s_1,\theta_1)=F(s_0,\theta_0)$, $k_0$, $k_1$ are the curvatures of the curve in $s_0$, $s_1$, respectively, and $l_{01}$ is the distance between the collision points $\gamma (s_0)$ and $\gamma(s_1)$. 

We remark that everything holds on a more general setting: billiard tables may not be convex or might have infinite horizon. In this general case, the billiard problem still defines a diffeomorphism preserving a probability measure, and its derivative can be implicitly computed.

We are mainly interested in the circular billiard, defined when $\gamma$ is the unitary circle. The billiard map is given then by the map $F: [0,2\pi) \times (0, \pi) \rightarrow [0,2\pi) \times (0, \pi)$ defined by $F(s_0,\theta_0)=(s_0+2\theta_0 \mod 2\pi,\theta_0)$. 

Each trajectory with outgoing angle $\theta_0$ corresponds to a caustic curve, which is a concentric circle with radius  $\cos \theta_0.$

\begin{figure}[!htb]
\centering
\includegraphics[scale=0.4]{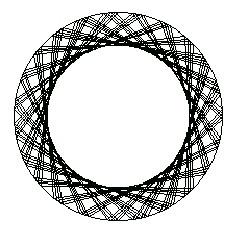}
\caption{Caustic of deterministic circular billiards for a trajectory.}
\end{figure}

Every horizontal circle $\{\theta = \theta_0\}$ in the phase space is invariant under $F$. Therefore, the phase space is foliated by these invariant horizontal circles. Furthermore, $F|_{\theta = \theta_0}$ is a circle rotation with an angle of $2\theta_0.$ 

We have, then, the following:
\begin{enumerate}
\item[a)] if the outgoing angle $\theta_0 = m\,\pi/n$ with $\gcd(m,n)=1,$ then for any $s_0\in[0,2\pi)$, $(s_0,\theta_0)$ is an $n$-periodic point; 
\item[b)] if the outgoing angle $\theta_0 = \beta \pi$ where $\beta$ is irrational, then the set $\{s_n : F^n(s_0,\theta_0) = (s_n,\theta_0)\}$ is dense in $[0,2\pi].$ That is, the projection in the first coordinate of the orbit of $(s_0,\theta_0)$ under $F$ is dense in the boundary circle;
\item[c)] if the outgoing angle $\theta_0 = \beta \pi$ with $\beta$ is irrational, then the billiard trajectory densely fills the circular ring formed by the boundary circle and the associated caustic. 
\end{enumerate}

\section{Random Map}
Let $(X, \mathcal{B}(X),\mu)$ be a measure space. For each $i = \{1, 2, \dots, N\}$, let $T_i: X \rightarrow X$ be a map, and let $p_i: X \rightarrow [0,1]$ be a function on $X$. Following \cite{Quas, lamb, feres, pelikan}, we define the random map $T:X \rightarrow X$, such that $T(x)= T_i(x)$ with probability $p_i(x)$. Iteratively, for each $n \in \mathbb{N}$, we define $T^{(n)}(x) = T_{i_n}\circ T_{i_{n-1}} \circ \dots \circ T_{i_1}(x)$ with probability  $p_{i_1}(x)p_{i_2}(T_{i_1}(x))\dots p_{i_n}(T_{i_{n-1}}\circ \dots \circ T_{i_1}(x)).$
The transition probability kernel of the random map $T$ is given by 
\begin{equation}
K(x, A) = \sum_{i=1}^{N} p_i(x)\mathds{1}_{A}(T_i(x)). \label{nucleo}
\end{equation}

The transition probability kernel $K(x,A)$, as shown in \eqref{nucleo}, defines the evolution of an initial distribution (probability measure) $\nu$ on $(X, \mathcal{B}(X))$ under the random map $T$ iteratively as
$$\nu^{(0)}:= \nu, \;\; \nu^{(n+1)}(A) = \int_X K(x,A)d\nu^{(n)}(x),$$
where $A \in \mathcal{B}(X)$ and $n \geq 1.$ 

\begin{definition}\label{medida invariante}
Let $T$ be a random map and let $\mu$ be a measure on $X.$ Then, $\mu$ is $T$-invariant if $\mu(A) = \sum_{i=1}^{N} \int_{T^{-1}_i(A)} p_i(x)d\mu(x)$ for all $A \in \mathcal{B}(X).$ 
\end{definition}

\subsection{The Feres Random Map}
Let us consider the Feres Random Map $T$ introduced in \cite{feres}. For a fixed $\alpha < \dfrac{\pi}{6},$ let $T: [0,\pi] \rightarrow [0,\pi]$ be such that 
\begin{equation}\label{M1}
T(\theta) = T_i(\theta) \mbox{ with probability } p_i(\theta),
\end{equation}
where $T_1(\theta) = \theta + 2\alpha,\; T_2(\theta) = -\theta + 2\pi - 4\alpha,\; T_3(\theta) = \theta -2\alpha$ and $T_4(\theta) = -\theta + 4\alpha$, and their respective probabilities are given by:
\begin{eqnarray*}
p_1(\theta)& = &\begin{cases} 1, & \;\;\;\;\;\;\;\;\;\;\;\;\;\;\;\;\;\;\;\;\;\;\; \mbox{if } \theta \in [0,\alpha) \\ u_{\alpha}(\theta), & \;\;\;\;\;\;\;\;\;\;\;\;\;\;\;\;\;\;\;\;\;\;\; \mbox{if }  \theta \in [\alpha, \pi -3\alpha) \\ 2\cos(2\alpha)u_{2\alpha}(\theta), & \;\;\;\;\;\;\;\;\;\;\;\;\;\;\;\;\;\;\;\;\;\;\; \mbox{if }  \theta \in [\pi - 3\alpha, \pi -2\alpha) \\ 0, & \;\;\;\;\;\;\;\;\;\;\;\;\;\;\;\;\;\;\;\;\;\;\; \mbox{if }  \theta \in [\pi-2\alpha, \pi]  \end{cases},\\
p_2(\theta)& = &\begin{cases} 0,  & \;\;\;\;\;\;\;\;\;\; \mbox{if } \theta \in [0,\pi -3\alpha) \\ u_{\alpha}(\theta)-2\cos (2\alpha)u_{2\alpha}(\theta),  & \;\;\;\;\;\;\;\;\;\; \mbox{if } \theta \in [\pi -3\alpha, \pi -2\alpha) \\ u_{\alpha}(\theta),  & \;\;\;\;\;\;\;\;\;\; \mbox{if } \theta \in [\pi - 2\alpha, \pi -\alpha) \\ 0, & \;\;\;\;\;\;\;\;\;\; \mbox{if } \theta \in [\pi-\alpha, \pi] 
\end{cases},\\
p_3(\theta)& = &\begin{cases} 0, & \;\;\;\;\;\;\;\;\;\;\;\;\;\;\;\;\;\;\;\; \mbox{if } \theta \in [0,2\alpha) \\ 2\cos (2\alpha)u_{2\alpha}(-\theta), & \;\;\;\;\;\;\;\;\;\;\;\;\;\;\;\;\;\;\;\; \mbox{if } \theta \in [2\alpha, 3\alpha) \\ u_{\alpha}(-\theta), & \;\;\;\;\;\;\;\;\;\;\;\;\;\;\;\;\;\;\;\; \mbox{if } \theta \in [3\alpha, \pi -\alpha) \\ 1, & \;\;\;\;\;\;\;\;\;\;\;\;\;\;\;\;\;\;\;\; \mbox{if } \theta \in [\pi-\alpha, \pi]
\end{cases},\\
p_4(\theta) & = &\begin{cases} 0, & \;\;\;\; \mbox{if } \theta \in [0,\alpha) \\ u_{\alpha}(-\theta), & \;\;\;\; \mbox{if } \theta \in [\alpha, 2\alpha) \\ u_{\alpha}(-\theta)-2\cos(2\alpha)u_{2\alpha}(-\theta), & \;\;\;\; \mbox{if } \theta \in [2\alpha, 3\alpha) \\ 0, & \;\;\;\; \mbox{if } \theta \in [3\alpha, \pi]  \end{cases},
\end{eqnarray*}
with $u_{\alpha}(\theta) = \frac{1}{2}(1+\frac{tg\alpha}{tg\theta}).$
For more details, see \cite{feres, lamb}. 

\begin{proposition}
Let $T$ be the Feres Random Map defined in \eqref{M1}. Then, the measure defined by $\mu(A)=\frac{1}{2}\int_A \sin(\theta)d\theta$ is $T$-invariant. 
\end{proposition}
\begin{proof}
See \cite{feres}.
\end{proof}

Consider $\Sigma = \{1,2,3,4 \}^{\mathbb{N}}$ the space of sequences formed by the symbols $1,2,3$ and $4$. With this in mind, we denote a sequence $(x_1,x_2,\dots)$ in $\Sigma$ by $\underline{x}.$

\begin{definition}\label{sigmatheta}
Given $\theta \in (0,\pi),$ we define 
$$\Sigma_\theta :=\{ \underline{x} \in \Sigma: p_{x_1}(\theta)p_{x_2}(T_{x_1}(\theta))\dots p_{x_k}\big( T_{x_{k-1}}\circ T_{x_{k-2}}\circ \dots \circ T_{x_1}(\theta)\big) >0, \;\; \forall k \geq 1  \}. $$ 
In other terms, a sequence $\underline{x}$ is in $\Sigma_{\theta}$ if the composition $T_{\underline{x}}^{(k)}(\theta):= T_{x_k} \circ T_{x_{k-1}} \circ \dots \circ T_{x_1}(\theta)$ has positive probability for each fixed $k \in \mathbb{N}.$
\end{definition}

\begin{definition} \label{admissivel}
Given $\theta \in (0,\pi)$, for every sequence $\underline{x}=(x_n)_{n} \in \Sigma_{\theta},$ the related sequence $(\theta_n)_n := (T_{\underline{x}}^{(n)}(\theta))_n$ is called an admissible sequence for $\theta.$      
\end{definition}


\begin{definition} \label{def1}
Given $\theta \in (0,\pi)$, we define $$\mathcal{C}(\theta) :=\left\{ \theta' \in (0,\pi) : T_{\underline{x}}^{(k)}(\theta)=\theta' \mbox{ for some } k\in \mathbb{N} \mbox{ and } \underline{x} \in \Sigma_\theta \right\}. $$
\end{definition}

Thus, $\mathcal{C}(\theta)$ is the set of all possible future images of the angle $\theta$ by the Feres Random Map $T.$ That is, $\theta' \in \mathcal{C}(\theta)$ if there are an admissible sequence $(\theta_n)_n$ and a $k \in \mathbb{N}$, such that $\theta_k = T_{x_k} \circ T_{x_{k-1}}\circ \dots \circ T_{x_1}(\theta) = \theta'.$

\begin{remark}
We observe that, in the case of the random maps $T_1 \circ T_3 = T_3 \circ T_1 = Id$, $T_2 \circ T_2 = Id$, and $T_4 \circ T_4 = Id$, if $\theta' \in \mathcal{C}(\theta)$, then $\theta \in \mathcal{C}(\theta')$. Therefore, $\mathcal{C}(\theta) = \mathcal{C}(\theta').$ So, given $\theta, \theta' \in (0,\pi),$ we conclude that  $\mathcal{C}(\theta) \cap \mathcal{C}(\theta') = \emptyset$ or $\mathcal{C}(\theta) = \mathcal{C}(\theta').$  \label{obs0} 
\end{remark}


\begin{proposition}
If $\alpha = \frac{m}{n}\pi$, with $m \in \mathbb{Z}$ and $n\in \mathbb{Z}^*$, then $\mathcal{C}(\theta)$ is finite.\label{inv}
\end{proposition}

\begin{proof}
Suppose that $\alpha = \frac{m}{n}\pi$ with $\gcd(m,n)=1.$ Consider $\tilde{T}_i \equiv T_i \mod \pi$, with $i=\{1,2,3,4 \}.$
We claim that $\{\tilde{T_i}\}_{i=1}^4$ generate the dihedral group. Indeed, notice that
\begin{equation*}
\begin{split}
& \tilde{T_1}^{n-1}(\theta) = \theta + 2(n-1)\alpha = \theta + 2n\alpha - 2\alpha \equiv \theta - 2\alpha \mod \pi = \tilde{T_3}(\theta), \\ & \tilde{T_2} \circ \tilde{T_1}^{n-4}(\theta)=-\theta -2n\alpha +2\pi +4\alpha \equiv -\theta + 4\alpha \mod \pi = \tilde{T_4}(\theta).
\end{split}
\end{equation*}
So we have that $G := < \tilde{T_1}, \tilde{T_2} >$ is the dihedral group generated by $\tilde{T_1}$ and $\tilde{T_2}$, and its order is $n$ or $\frac{n}{2}$, depending on the parity of $n.$ Finally, notice that $\mathcal{C}(\theta)$ is contained in the orbit of $\theta$ by the action of the group $G$, that is, the order of $\mathcal{C}(\theta)$ is less than or equal to the order of $G$. Therefore, $\mathcal{C}(\theta)$ is finite. 
\end{proof}

\begin{proposition}
If $\alpha=\beta\pi$ with $\beta \in \mathbb{R}\setminus \mathbb{Q},$ then $ \mathcal{C}(\theta)$ is countably infinite.\label{infinito}
\end{proposition}

\begin{proof}
Given $\theta' \in \mathcal{C}(\theta)$, we have that $\theta' = \pm \theta + J\pi + K\alpha$ for some $J,K \in 2\mathbb{Z}$. If all $\theta'\in \mathcal{C}(\theta)$ are distinct, then $\mathcal{C}(\theta)$ is countably infinite. We observe that $$\theta +J_1\pi+K_1\alpha = \theta +J_2\pi +K_2\alpha $$ if $\alpha = \frac{J_2-J_1}{K_1-K_2}\pi.$ Since $\frac{\alpha}{\pi}$ is an irrational number, all elements in $\mathcal{C}(\theta)$ that are of the form $\theta +J\pi+K\alpha$ are distinct. Therefore, $\mathcal{C}(\theta)$ is countably infinite. 
\end{proof}

\begin{lemma} 
Given $\theta \in (0,\pi)$, such that $\theta \neq k\alpha$ with $k\in\mathbb{N},$ there are an admissible sequence $(\theta_n)_n$ and an $m \in \mathbb{N}$, such that $\theta_m \in (0,\alpha).$\label{cteta} 
\end{lemma}
\begin{proof}
Suppose initially that $\theta \in (\alpha,2\alpha).$ We can apply $T_4$, and so, $T_4(\theta) \in (2\alpha, 3\alpha).$ Then, $T_3 \circ T_4(\theta) \in (0,\alpha).$ Suppose now that $\theta$ belongs to a subinterval other than $(\alpha,2\alpha).$ There is a maximum number of times that we can apply $T_3$ consecutively; let $l$ be this number. Thus, we have two options: $T_3^l(\theta) \in (0,\alpha)$ or $T_3^l(\theta) \in (\alpha,2\alpha).$ If $T_3^l(\theta) \in (\alpha, 2\alpha)$, then it follows that $T_3 \circ T_4 \circ T_3^l(\theta) \in (0,\alpha).$    
\end{proof}

\subsection{The Strong Knudsen's Law for the Feres Random Map}

Following \cite{Quas, lamb}, we will define a skew-product in the space $\Omega = [0,1] \times [0, \pi]$. Let $T : [0,\pi] \rightarrow [0,\pi]$ be the Feres Random Map defined in \eqref{M1}, and let $$J_k = \Big\{(x, \theta) \in \Omega : \sum_{i<k}p_i(\theta) \leq x < \sum_{i \leq k}p_i(\theta)\Big\},$$ with $k=1,2,3,4.$
Consider $\varphi_k(x,\theta) = \frac{1}{p_k(\theta)}\big(x - \sum_{i=1}^{k-1}p_k(\theta)\big)$, if $(x, \theta) \in J_k$. We define the skew-product of $T$ by the map $S:\Omega \rightarrow \Omega$, such that:
\begin{equation}\label{skew1}
S(x,\theta) = (\varphi_k(x, \theta), T_k(\theta)), \; \mbox{if} \; (x, \theta) \in J_k.
\end{equation}

The next theorem is proved in \cite{lamb}.

\begin{theorem}\label{teo lamb}
Let $T$ be the Feres Random Map described in \eqref{M1}, let $\mu$ be the measure defined by $\mu(A)=\frac{1}{2}\int_A\sin(\theta)d\theta$, and let $\nu \ll \mu.$ If $\frac{\alpha}{\pi}$ is an irrational number, then $\nu^{(n)}(A) \rightarrow \mu(A)$, for all $A \in \mathcal{B}([0,\pi])$.
\end{theorem}

\begin{proof}
We provide only a sketch of the proof; for more details, see \cite[Equation (7.2) and Theorem 14]{lamb}.
Let $U_S: \mathcal{L}^P(\Omega)\rightarrow\mathcal{L}^P(\Omega)$ be the Koopman operator associated with the skew-product $S$, defined in \eqref{skew1}, that is, $U_S(g) = g\circ S$. Let $f \in L^1 ([0, \pi], \mu)$ be the Radon-Nikodym derivative, and let $\nu \ll \mu$. Then,
\begin{equation}\label{conv}
\nu^{(n)}(A) =\int U_S^{n}(\mathds{1}_{\pi_2^{-1}(A)})d(\lambda\times \nu) = \int \pi^*_2(f) U_S^{n}(\mathds{1}_{\pi_2^{-1}(A)})d(\lambda\times \mu) \rightarrow \mu(A),
\end{equation}
if $S$ is mixing or is an exact endomorphism. In fact, \cite[Theorem 14]{lamb} proves that $S$ is an exact endomorphism if $\frac{\alpha}{\pi}$ is an irrational number. 
\end{proof}
The convergence  $\nu^{(n)}(A) \rightarrow \mu(A)$ is called the Strong Knudsen's Law for the random map $T.$ 

\begin{remark}\label{obspipeline}
Suppose that the Strong Knudsen's Law holds for $T$. Given an initial distribution $\nu$, the angle distribution after an arbitrarily large number of collisions along the pipeline is close to to the uniform distribution $\mu$.
\end{remark}

We present at this moment a proposition that completes the Theorem \ref{teo lamb}. That is, with this proposition we have a necessary and sufficient condition for the Strong Knudsen's Law for the random map we are working on.

\begin{proposition}\label{lei racional}
Let T be the Feres Random Map, let $\mu(A)=\frac{1}{2}\int_A \sin(\theta)d\theta$, and let $\nu \ll \mu$. If $\frac{\alpha}{\pi}$ is an rational number, then the Strong Knudsen's Law does not hold.
\end{proposition}

\begin{proof}
Given $\alpha=\frac{m \pi}{n}$, with $\gcd(m,n)=1,$ we want to show that there is a set $B \subset \Omega$ such that $B$ is $S$-invariant and $0<(\lambda \times \mu) (B) < 1$. For this, define the following sets: $A_1=\big\{0,\frac{\pi}{n}, \frac{2\pi}{n}, \cdots, \frac{m \pi}{n}, \cdots, \frac{(n-1)\pi}{n}, \pi \big\}$ and $A_2=\big\{ \frac{\pi}{2n}, \frac{3\pi}{2n},\cdots, \frac{(2n-1)\pi}{2n} \big\}$. We observe that these sets $A_1$ and $A_2$ are invariant by the Feres Random Map, that is, if $x \in A_i,$ with $i=1,2$, then $T_k(x) \in A_i$, whatever $k=1,2,3,4$ is, since $x$ belongs to the domain of $T_k$. In fact:
\begin{itemize}
\item If $\frac{(2j-1)\pi}{2n}\in A_2$ is in the domain of $T_1$, then
$$T_1\Big(\frac{(2j-1)\pi}{2n}\Big) = \frac{(4m+2j-1)\pi}{2n} \in A_2;$$

\item If $\frac{(2j-1)\pi}{2n}\in A_2$ is in the domain of $T_2$, then
 $$T_2\Big(\frac{(2j-1)\pi}{2n}\Big) = \frac{(4n-8m-2j+1)\pi}{2n} \in A_2;$$

\item If $\frac{(2j-1)\pi}{2n}\in A_2$ is in the domain of $T_3$, then 
$$T_3\Big(\frac{(2j-1)\pi}{2n}\Big) = \frac{(2j-4m-1)\pi}{2n} \in A_2;$$

\item If $\frac{(2j-1)\pi}{2n}\in A_2$ is in the domain of $T_4$, then 
$$T_4\Big(\frac{(2j-1)\pi}{2n}\Big) = \frac{(8m-2j+1)\pi}{2n} \in A_2.$$
\end{itemize} 

Similarly, it can be shown that the set $A_1$ is invariant by the Feres Random Map. Taking $\epsilon=\frac{\pi}{4n},$ let us consider the interval $I_j=\Big(\frac{(2j-1)\pi}{2n}-\epsilon,\frac{(2j-1)\pi}{2n}+\epsilon \Big)$, with $j=1,2,\dots, n$. Therefore, the set $\cup_{j=1}^n I_j$ is invariant by the Feres Random Map, since the set $A_2$ is invariant by this map, and the maps $T_k$ are translations in $[0,\pi]$. That is, given $k \in \{1,2,3,4\}$ and $j\in \{1,2,\dots,n\}$, we have that $T_k(I_j)=I_l$, for some $l\in \{1,2,\dots,n\}.$

We also observe that the Lebesgue measure of the set $\cup_{j=1}^n I_j$ in $[0,\pi]$ satisfies $Leb\Big( \cup_{j=1}^n I_j \Big)=\frac{\pi}{2} <\pi= Leb([0,\pi])$. Finally, the set $[0,1] \times \cup_{j=1}^n I_j \subset \Omega,$ is an $S$-invariant set; and, since $\lambda \times \mu$ is absolutely continuous with respect to the Lebesgue measure in $\Omega$, it follows that $0 < (\lambda \times \mu)([0,1] \times \cup_{j=1}^n I_j) < 1$. That is, $S$ is not ergodic and, then, $S$ is not an exact endomorphism.

Therefore, for $\alpha=\frac{m\pi}{n}$, with $\gcd(m,n)=1$, we have that $\nu^{(n)} \not\to \mu$, since this convergence occurs, if and only if, $S$ is an exact endomorphism.
\end{proof}


With both Theorem \ref{teo lamb} and the Proposition \ref{lei racional}, we obtain the following theorem:

\begin{theorem} \label{lamb final}
Let $T$ be the Feres Random Map, let $\mu(A)=\int_A\sin(\theta)d\theta$, and let $\nu \ll \mu$. Then, $\nu^{(n)}(A) \to \mu(A)$ for all $A \in \mathcal{B}([0,\pi])$, if and only if, $\frac{\alpha}{\pi}$ is an irrational number. 
\end{theorem}

\section{The Random Billiard Map}\label{sec4}
Let $\overline{T}:[0,L] \times (0,\pi) \rightarrow [0,L]\times (0,\pi)$, where $\overline{T}(s,\theta)=(s,T_i(\theta))$ with probability $p_i(\theta)$, be the random map that extends the Feres Random Map $T(\theta) = T_i(\theta)$ with probability $p_i(\theta).$ Consider a curve in the Euclidean plane and the corresponding deterministic billiard map $F: [0,L)\times (0,\pi) \rightarrow [0,L)\times (0,\pi)$ associated with this curve. We have that $F(s,\theta)=(s_1(s,\theta), \theta_1(s,\theta))$ preserves the measure $\lambda \times \mu$, where $\lambda$ is the normalized Lebesgue measure at $[0, L)$, and $\mu(A)=\frac{1}{2}\int_A\sin(\theta)d\theta.$

\begin{definition}
Consider a deterministic billiard map $F$. We define the random billiard map $\overline{F}: [0,L) \times (0,\pi)\rightarrow [0,L)\times (0,\pi)$ by
\begin{equation}\label{random}
\overline{F}(s,\theta)= F\circ\overline{T}(s,\theta)=\Big(s_1(s,T_i(\theta)),\theta_1(s,T_i(\theta))\Big)
\end{equation}
with probability $\overline{p}_i(s,\theta)=p_i(\theta).$
\end{definition}

\begin{example}
Let $F(s,\theta)=(s+2\theta\mod 2\pi,\theta)$ be the deterministic billiard map in the circle. Then, the random billiard map in the circle is $\overline{F}(s,\theta)= F \circ \overline{T}(s,\theta)=(s+2T_i(\theta) \mod 2\pi, T_i(\theta))$ with probability $p_i(\theta).$
\end{example}

We observe that the projection in the second coordinate of the iterates of the map $S$ represents the outgoing angles of the particle moving in a pipeline. Thus, on both billiard tables (infinite pipeline and circle), we have the same behavior of the orbits referring to the angles (see Figure \ref{pipeline}). That is, the outgoing angles do not depend on the collision points, and are given by the compositions of the $T_i$ maps. 

\begin{figure}[!htb]
\centering
\includegraphics[scale=0.2]{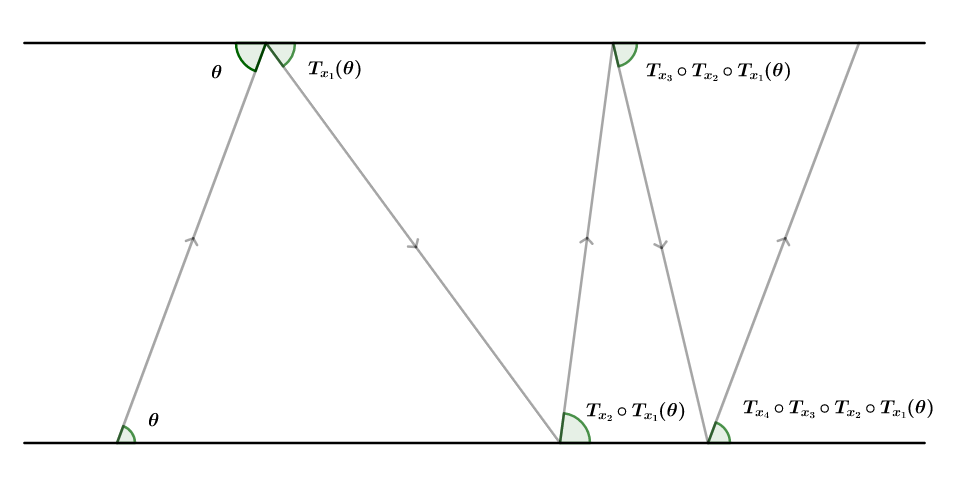}
\includegraphics[scale=0.15]{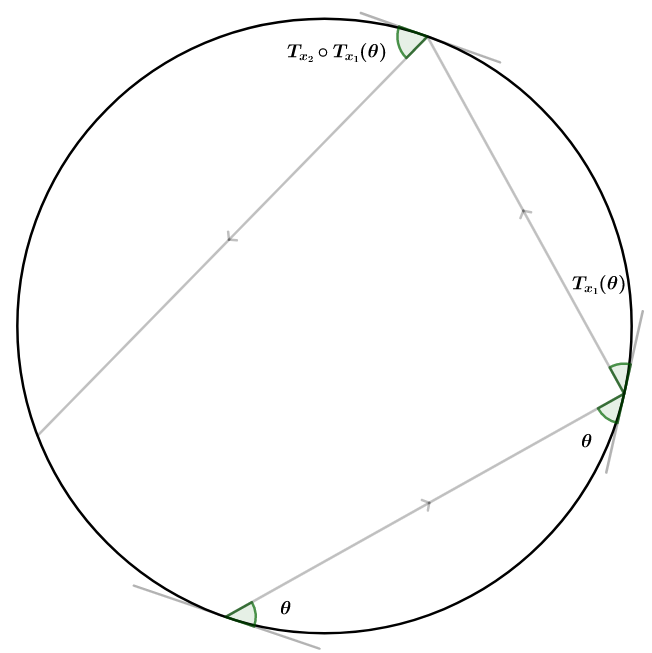}
\caption{Left, it is a random billiard in the infinite pipeline. Right, it is a random billiard in the circle.}\label{pipeline}
\end{figure}

So we can enunciate Theorem \ref{lamb final} in the language of random billiards as follows:

\begin{theorem}\label{prop4.4}
Consider a particle moving in a circle (or infinite pipeline) and let $T$ be the Feres Random Map. Given a measure $\nu \ll \mu$, where $\mu(A) = \frac{1}{2}\int \sin(\theta)d\theta$, we have that $\nu^{(n)}(A) \to \mu(A)$ for all $A \in \mathcal{B}([0,\pi])$, if and only if, $\frac{\alpha}{\pi}$ is an irrational number.
\end{theorem}

\begin{proof}
Consider the random billiard map $\overline{F}$ in the circle or the infinite pipeline. Notice that the projection on the second coordinate of the iterates of the random billiard map $\overline{F}$ represents the Feres Random Map since the outgoing angles in these two cases do not depend on the particle's position (see Figure \ref{pipeline}).  So, by Theorem \ref{lamb final}, we have the Strong Knudsen's Law if and only if $\frac{\alpha}{\pi}$ is an irrational number. 
\end{proof}

Remark \ref{obspipeline} also applies to the particle colliding on a circular table.


\subsection{Invariant Measure For The Random Billiard Map} \label{secmedida}

Given a deterministic billiard map $F$, the associated random billiard map is given by $\overline{F}(s,\theta) = F(s,T_i(\theta))$ with probability $p_i(\theta)$. Consider $\overline{T}_i(s,\theta)=(s,T_i(\theta))$ and $\overline{F}_i(s,\theta)=F\circ \overline{T}_i(s,\theta)$, for each $i=1,\dots,4$. By Definition \ref{medida invariante}, a measure $\nu$ is invariant for the random billiard map $\overline{F}$ if $$\nu(A)  = \sum_{i=1}^4 \iint _{\overline{F}^{-1}_i(A)}p_i(\theta)d\nu(s,\theta)$$ for every Borel set $A$ in $(0,L) \times (0,\pi).$

\begin{lemma}\label{Ti}
The measure $\lambda \times \mu$ preserved by deterministic billiard map $F$ is an invariant measure for the random map $\overline{T}(s,\theta)=(s, T_i(\theta))$ with probability $p_i(\theta).$
\end{lemma}

\begin{proof}
The measure $\lambda \times \mu$ is invariant for $\overline{T}$ if $$\sum_{i=1}^4 \int \int \mathds{1}_{A\times B}(\overline{T_i}(s,\theta))p_i(\theta)d\lambda(s)d\mu(\theta) = (\lambda \times \mu)(A \times B)$$
for all $A\times B \in \mathcal{B}([0,L]\times [0,\pi]).$ Indeed, notice that   
\begin{align*}
 \sum_{i=1}^4 \int \int \mathds{1}_{A \times B}({\overline{T_i}}(s,\theta))p_i(\theta)d\lambda(s) d\mu(\theta)  
& = \sum_{i=1}^4 \int \int \mathds{1}_{A}(s) \mathds{1}_{B}(T_i(\theta)) p_i(\theta)d\lambda(s) d\mu(\theta)
\\ & = \lambda(A) \sum_{i=1}^4 \int \mathds{1}_{B}(T_i(\theta))p_i(\theta) d\mu(\theta) 
\\ & =  \lambda(A)\mu(B).
\end{align*}
\end{proof}

\begin{proposition}\label{prop4.5}
Given a deterministic billiard map $F$ with an invariant measure $\lambda \times \mu$, and the random map $\overline{T}(s,\theta)=(s, T_i(\theta))$ with probability $p_i(\theta)$, we have that the random billiard map $\overline{F}$ has $\lambda \times \mu$ as an invariant measure.
\end{proposition}

\begin{proof}
By Lemma \ref{Ti}, $(\lambda \times \mu)(A \times B) = \sum_{i=1}^N \int \int \mathds{1}_{\overline{T_i}^{-1}(A \times B)}(s,\theta)p_i(\theta) d\lambda \times \mu.$ Due to the invariance of the measure $\lambda \times \mu$ for the deterministic billiard map $F,$ we have that $(\lambda \times \mu)(F^{-1}(A \times B)) = (\lambda \times \mu)(A \times B)$. Therefore,
\begin{align*}
 \sum_{i=1}^4 \int \int \mathds{1}_{\overline{F}_i^{-1}(A \times B)}(s,\theta)p_i(\theta)d\lambda(s) d\mu(\theta)  & = \sum_{i=1}^4 \int \int \mathds{1}_{\overline{T_i}^{-1}\circ F^{-1}(A \times B)} p_i(\theta)d\lambda(s) d\mu(\theta)
\\ & = (\lambda \times \mu)(F^{-1}(A \times B))
\\& = (\lambda \times \mu)(A \times B).
\end{align*}
\end{proof}

\begin{remark}
We can also think of random billiard maps as $\overline{F}(s,\theta) = \overline{T} \circ F(s,\theta)$ with probability $\overline{p}_i(s,\theta) = p_i(F(s,\theta))=p_i(\theta_1(s,\theta))$. We show next that this composition also has the same invariant measure $\lambda \times \mu$ as the deterministic billiard maps.
\end{remark}

\begin{proposition} \label{prop4.6}
Let $F$ be a deterministic billiard map with an invariant measure $\lambda \times \mu$, and let $\overline{T}(s,\theta)=(s, T_i(\theta))$ be the random map with probability $p_i(\theta)$. Then, the random billiard map $\overline{F}(s,\theta)=\overline{T}_i \circ F(s,\theta)$ with probability $p_i(F(s,\theta))$ has $\lambda \times \mu$ as an invariant measure.
\end{proposition}

\begin{proof}
By Lemma \ref{Ti}, $$(\lambda \times \mu)(\varphi) = \sum_{i=1}^n \int\int p_i(\theta) \cdot \varphi \circ \overline{T}_i(s,\theta) \; d\lambda(s) \; d\mu(\theta)$$ for all $\varphi$ integrable.
Then, we have that  
\begin{align*}
& \sum_{i=1}^4 \int \int \overline{p}_i \cdot \varphi \circ (\overline{T}_i \circ F) \; d\lambda \; d\mu = \sum_{i=1}^4 \int \int p_i \circ F \cdot \varphi \circ (\overline{T}_i \circ F) \; d\lambda \; d\mu   
\\ & = \sum_{i=1}^4 \int \int \big(p_i \cdot (\varphi \circ \overline{T}_i) \big) \circ F \; d\lambda \; d\mu = \sum_{i=1}^4 \int \int p_i \cdot \varphi \circ \overline{T}_i \; d\lambda \; d\mu
\\ & = (\lambda \times \mu)(\varphi).
\end{align*}  
Thus, the measure $\lambda \times \mu$ is invariant for the random billiard map $\overline{F}$.   
\end{proof}

Propositions \ref{prop4.5} and \ref{prop4.6}, ensure that the measure preserved by the deterministic billiard map is an invariant measure for the random billiard map for both compositions: $\overline{F}(s,\theta) = F \circ \overline{T}_i(s,\theta)$ with probability $p_i(\theta)$, and $\overline{F}(s,\theta) = \overline{T}_i\circ F(s,\theta)$ with probability $p_i(F(s,\theta))$.

\section{Markov Chain}
The random map $T$, defined in \eqref{M1}, can be seen as a Markov chain with state space $\mathcal{C}(\theta)$, and the transition probabilities are those generated by the respective probabilities $p_1$, $p_2$, $p_3$, and $p_4$.

\begin{example}
Let $\alpha = \frac{\pi}{7}$, and let $\theta \in (0,\alpha)$. The state space for the Markov chain is $\mathcal{C}(\theta)=\{\theta, \theta+2\alpha, \theta+4\alpha, \theta+6\alpha, -\theta+2\alpha, -\theta+4\alpha, -\theta+6\alpha \}$. 

\begin{figure}[!htb]
\centering
\includegraphics[scale=0.2]{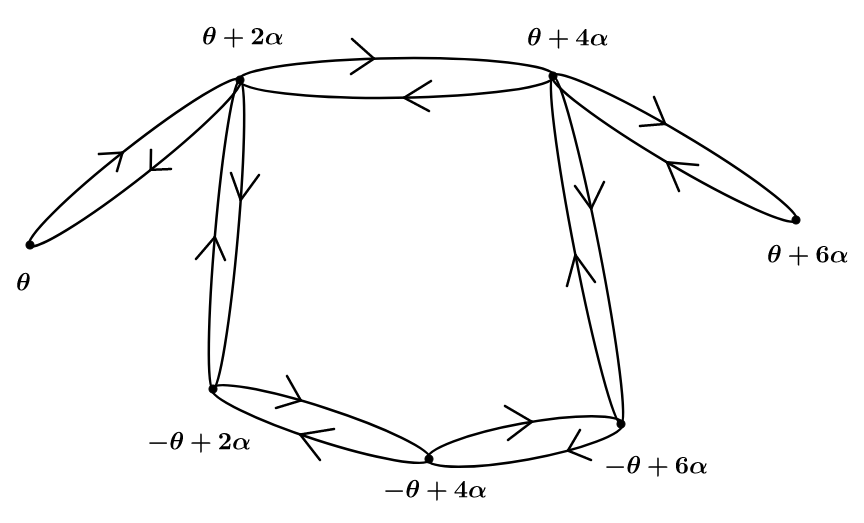}
\caption{Markov chain whose vertices are the images of $\mathcal{C}(\theta)$ with $\alpha=\frac{\pi}{7}$ and $\theta \in (0,\alpha).$}\label{fig2}
\end{figure}

Moreover, the transition probability matrix for this Markov chain is:
{\tiny
$$A = \left[
\begin{array}{ccccccc}
0 & p_1(\theta) & 0 & 0 & 0 & 0 & 0 \\
p_3(\theta+2\alpha) & 0 & p_1(\theta+2\alpha) & 0 & p_4(\theta+2\alpha) & 0 & 0 \\
0 & p_3(\theta+4\alpha) & 0 & p_1(\theta+4\alpha) & 0 & 0 & p_2(\theta+4\alpha) \\
0 & 0 & p_3(\theta+6\alpha) & 0 & 0 & 0 & 0 \\
0 & p_4(-\theta+2\alpha) & 0 & 0 & 0 & p_1(-\theta+2\alpha) & 0 \\
0 & 0 & 0 & 0 & p_3(-\theta+4\alpha) & 0 & p_1(-\theta+4\alpha) \\
0 & 0 & p_2(-\theta+6\alpha) & 0 & 0 & p_3(-\theta+6\alpha) & 0
\end{array}
\right].$$}
\end{example}

Let $(X,\mathcal{A})$ be a measurable space with $X=\{ x_1, x_2, \dots , x_r, \dots \}$ finite or countably infinite. Consider $\Sigma = X^{\mathbb{N}}$ the space of the sequences in $X$ endowed with the $\sigma$-algebra product. Let $\sigma: \Sigma \rightarrow \Sigma$ be the shift operator, that is, such that $\sigma(x_n)_n = (x_{n+1})_n$. Given a family of transition probabilities $\{p_{x_ix_j} : i,j =1,2,\dots,r, \dots \},$ we define the transition matrix $(p_{x_ix_j})$ with $\sum_j p_{x_ix_j} = 1$. This matrix is called the stochastic matrix, and we denote it by $P$.

A measure $\mu$ in $X$ is completely characterized by the values $\mu_i = \mu(x_i)$, with\break $i \in \{1,2, \dots, r, \dots \} .$ We say that $\mu$ is a stationary measure for the Markov chain if it satisfies $\sum_{i} \mu_i p_{x_ix_j} = \mu_j$, for every $j$.
If the stationary measure $\mu$ is a probability, then we say that $\mu$ is a stationary distribution for the Markov chain.

We define the Markov measure as $\nu (m; x_m, \dots, x_n) = \mu_{x_m} p_{x_m, x_{m+1}} \dots p_{x_{n-1}, x_n}.$
Consider the set $\Sigma_P$ of all sequences $\underline{x}=(x_n)_n$ in $\Sigma$ such that $p_{x_n, x_{n+1}} > 0$, for all $n \in \mathbb{N}$. We say that the sequences $\underline{x} \in \Sigma_P$ are admissible sequences.  

\begin{definition}
The stochastic matrix $P$ is irreducible if, for all $x_i, x_j$, there is an $n > 0$ such that $p^n_{x_i, x_j} > 0.$
\end{definition}

In other words, $P$ is irreducible if it is possible to move from any state $x_i$ to any state $x_j$ in a certain number $n$ of steps that depends on $i$ and $j.$

\begin{theorem}
Let $P$ be the stochastic matrix of a Markov chain.
\begin{enumerate}
\item[(i)]The Markov shift $(\sigma, \nu)$ with finite symbols is ergodic, if and only if the stochastic matrix $P$ is irreducible.\label{markov}
\item[(ii)]The Markov shift $(\sigma, \nu)$ with countably infinite symbols is ergodic, if and only if there is a stationary distribution $\mu$, and $P$ is irreducible.  \label{shift enumeravel} 
\end{enumerate}
\end{theorem}

\begin{proof}
For $(i)$, see \cite{marceloviana}, and for $(ii)$, see \cite[Example 7.1.7]{Durrett}.
\end{proof}

\begin{proposition}\label{markovrandom}
Given $\theta \in (0,\pi)$, define $\Sigma = \cup_{\theta' \in \mathcal{C}(\theta)} \Sigma_{\theta'}$ with $\Sigma_{\theta'}$ as in the Definition \ref{sigmatheta}. If $\frac{\alpha}{\pi}$ is a rational number, then the Markov shift in $\Sigma$ is ergodic. Moreover, if $\frac{\alpha}{\pi}$ is an irrational number and the Markov chain is aperiodic, then the Markov shift in $\Sigma$ is ergodic.
\end{proposition}

\begin{proof}
Initially, suppose that $\frac{\alpha}{\pi}$ is a rational number. Observe that Proposition \ref{inv} implies that $\mathcal{C}(\theta)$ is finite. Hence, by Theorem \ref{markov}, we need to verify that the stochastic matrix $P$ in $\mathcal{C} (\theta)$ is irreducible. Notice that if $\theta', \theta'' \in \mathcal{C}(\theta)$, then there is an admissible sequence $(\theta_n)_n $ and also a $k \in \mathbb{N}$, such that $\theta_0 = \theta'$, $\theta''=T_{x_k}\circ \dots \circ T_{x_1}(\theta')$ and  $p_{\theta' \theta''}^k = p_{x_1}(\theta')p_{x_2}(T_{x_1}(\theta'))\dots p_{x_k}\big( T_{x_{k-1}}\circ T_{x_{k-2}}\circ \dots \circ T_{x_1}(\theta')\big) > 0.$ Therefore, any state $\theta' \in \mathcal{C}(\theta)$ has a positive probability of reaching any other state $\theta'' \in \mathcal{C}(\theta)$. So, the stochastic matrix is irreducible, and we conclude that the Markov shift in $\Sigma$ is ergodic.

From now on, suppose that $\frac{\alpha}{\pi}$ is an irrational number and $\theta \in (0,\pi)$. By Proposition \ref{infinito}, the set $\mathcal{C}(\theta)$ is countably infinite. Similarly to the first case, for any $\theta'$ and $\theta''$ in $\mathcal{C}(\theta),$ there is a $k$, such that $p_{\theta' \theta''}^k > 0$. By Theorem \ref{markov}, the Markov chain admits a stationary distribution. Therefore, the Markov shift in $\Sigma$ is ergodic.  
\end{proof}

\section{The Random Billiard Map in the Circle}
Consider the deterministic billiard map $F(s,\theta)=(s+2\theta\mod 2\pi,\theta)$ in the circle. Thus, the random billiard map in the circle is $\overline{F}(s,\theta)= F \circ \overline{T_i}(s,\theta)=(s+2T_i(\theta) \mod 2\pi, T_i(\theta))$ with probability $p_i(\theta).$

\subsection{The Strong Knudsen's Law For the Random Billiard Map in the Circle}
Given a random billiard map $\overline{F}(s,\theta) = F(s,T_i(\theta))$ with probability $p_i(\theta)$, we saw in Section \ref{secmedida} that the random billiard map preserves the measure $\lambda \times \mu$, that is, $(\lambda \times \mu)(A\times B) = \sum_{i=1}^N\iint \mathds{1}_{A\times B}(\overline{F}_i(s,\theta))p_i(\theta)d\lambda (s) d\mu (\theta)$. We define the transition probability kernel by
$$K((s,\theta), A\times B) = \sum_{i=1}^4 p_i(\theta)\mathds{1}_{A\times B}(\overline{F}_i(s,\theta)).$$  
This transition probability kernel defines the evolution of an initial distribution $\nu$ in $([0,\pi] \times [0,L], \mathcal{B})$ under $\overline{F}$ iteratively as $\nu^{(0)}=\nu$ and
$$\nu^{(n+1)}(A\times B) = \iint K((s,\theta), A\times B) d\nu^{n}(s,\theta)$$
for $n \geq 1$.

\begin{proposition}\label{leiforte}
Let $\frac{\alpha}{\pi}$ be an irrational number and let $\overline{F}$ be the random billiard map in the circle. If $\nu(s,\theta) := \nu_1(s) \times \nu_2(\theta) = \lambda(s) \times g(\theta)\mu(\theta),$ then
$$\nu^{(n)}(A \times B) = \nu_1^{(n)}\times\nu_2^{(n)} (A \times B) \longrightarrow (\lambda \times \mu)(A\times B).$$
\end{proposition}

\begin{proof}
Observe that $\nu_2(\theta) = g(\theta)\mu(\theta)$ is absolutely continuous with respect to $\mu$. In addition, we have that
\begin{align*}
\nu^{(1)}(A \times B) & = \iint K((s,\theta), A\times B) d\nu (s,\theta) 
\\ &  
= \sum_{i=1}^4 \int \int \mathds{1}_{A \times B}(\overline{F}_i(s,\theta)) p_i(\theta)d(\lambda(s) g(\theta)\mu(\theta))
\\ & =
\sum_{i=1}^4 \int p_i(\theta)\mathds{1}_{B}(T_i(\theta)) \Big(\int \mathds{1}_{A}(s+2T_i(\theta) \mod 2\pi) d\lambda(s)\Big) d(g(\theta)\mu(\theta))
\\ & =
\lambda(A) \sum_{i=1}^4 \int \mathds{1}_B(T_i(\theta))p_i(\theta) d(g(\theta)\mu(\theta)) 
\\ & =
\lambda(A) \times \nu_2^{(1)}(B). 
\end{align*}
Inductively, we obtain that $\nu^{(n)}(A \times B) = \lambda(A) \times \nu_2^{(n)}(B)$. By Theorem \ref{prop4.4}, we have that $\nu_2^{(n)}(B) \rightarrow \mu(B)$ and therefore $\nu^{(n)}(A\times B) \rightarrow (\lambda \times \mu)(A \times B)$. In other words, we proved the Knudsen's Strong Law for $\overline{F}$ for a particular family of measures $\nu$.
\end{proof}

The geometric meaning of Proposition \ref{leiforte} is as follows: given an irrational number $\frac{\alpha}{\pi}$, the distribution of points and angles after an arbitrarily large number of collisions in the circle is close to the uniform distribution $\lambda \times \mu$.

\subsection{Dense Orbits}
In this section, we show the abundance of trajectories that are dense at the boundary of the circular table as well as in the ring between the boundary of the billiard table and random caustics, regardless whether $\frac{\alpha}{\pi}$ is rational or irrational.

Given $(s,\theta) \in [0,2\pi)\times (0,\pi)$ and $\underline{x} \in \Sigma_\theta$, we say that the orbit of $(s,\theta)$ by $\overline{F}$ is dense in the circle if the sequence $(s+2\sum_{k=1}^nT_{x_k} \circ T_{x_{k-1}}\circ \dots \circ T_{x_1}(\theta) \mod 2\pi)_n$ is dense in $[0,2\pi)$, that is, if the projection in the first coordinate of the iterates of $\overline{F}_{\underline{x}}(s,\theta)$ is dense in $[0,2\pi)$.

\begin{proposition}\label{densa}
Let $\alpha=\beta \pi$ with $\beta \in (0,\frac{1}{6}).$ Then, for all $\theta \in (0,\pi)$ such that $\frac{\theta+\alpha}{\pi}$ is an irrational number, there is a dense orbit for the random billiard map in the circle.  
\end{proposition}

\begin{proof}
For every $\theta \in (0,\pi-2\alpha)$, notice that $p_1(\theta)p_3(T_1(\theta))>0$; therefore, the sequence $\underline{x} = (131313\dots)$ is in $\Sigma_\theta$, and also generates a dense orbit in the circle. Indeed, observe that $\overline{F}_{\underline{x}}^{(2n)}(s,\theta)=(s+4n(\theta+\alpha) \mod 2\pi,\theta)$, that is, at even times, the random billiard behaves like a deterministic billiard of initial angle $4(\theta + \alpha)$. Thus, if $\frac{\theta+\alpha}{\pi}$ is an irrational number, then the orbit is dense in the circle. Similarly, if $\theta \in (2\alpha,\pi)$, we have that $p_3(\theta)p_1(T_3(\theta)) >0$ and, therefore, the sequence $\underline{x}=(313131\dots) \in \Sigma_\theta$ generates a dense orbit in the circle.
\end{proof}



In order to enunciate the next few results, notice that we can relate each sequence $\underline{y} \in \big(\mathcal{C}(\theta) \big)^\mathbb{N}$ to a unique sequence $\underline{x} \in \cup_{\theta' \in \mathcal{C}(\theta)} \Sigma_{\theta'}$ and vice versa. Thus, for each $\underline{y} \in \big(\mathcal{C}(\theta) \big)^\mathbb{N}$, consider the related trajectory given by the sequence $\underline{x} \in \cup_{\theta' \in \mathcal{C}(\theta)} \Sigma_{\theta'}$, that is, consider $\overline{F}_{\underline{x}}.$

\begin{theorem}\label{orbitadensa}
Let $\overline{F}:[0,2\pi)\times (0,\pi) \to [0,2\pi)\times (0,\pi)$ be the random billiard map in the circle. Given $(s,\theta) \in [0,2\pi) \times (0,\pi)$, let $\nu$ be the Markov measure in $\big( \mathcal{C}(\theta) \big)^\mathbb{N}.$ 
\begin{enumerate}
\item[1.] If $\frac{\alpha}{\pi}$ is a rational number and $\frac{\theta}{\pi}$ is an irrational number, then for $\nu$-almost every $\underline{y}\in \big( \mathcal{C}(\theta) \big)^\mathbb{N}$ the related trajectory $\overline{F}_{\underline{x}}(s,\theta)$ is dense in the circle.
\item[2.] If $\frac{\alpha}{\pi}$ is a irrational number and the Markov chain is aperiodic in $\mathcal{C}(\theta)$, then for $\nu$-almost every $\underline{y}\in \big( \mathcal{C}(\theta) \big)^\mathbb{N}$ the related trajectory $\overline{F}_{\underline{x}}(s,\theta)$ is dense in the circle.
\end{enumerate}
\end{theorem}

\begin{proof}
By Proposition \ref{markovrandom}, we have the ergodicity of the Markov shift in both items.  

For the first part, fix the open interval $(a,b) \subset [0,2\pi]$, let $\frac{\theta}{\pi}$ be an irrational number, and consider $\theta' \in \mathcal{C}(\theta)$ where $T_3 \circ T_1(\theta')$ satisfies $p_3(T_1(\theta'))p_1(\theta')>0$. By Proposition \ref{densa}, there is an $N \in \mathbb{N}$ such that the finite sequence $\underline{z}=(1313 \dots 13)\in \Sigma_{\theta'}$ with $N$ elements satisfies $\overline{F}^N_{\underline{z}}(s,\theta')\cap \big((a,b)\times(0,\pi)\big)\neq\emptyset$ for all $s \in [0,2\pi]$, since the sequence $\underline{x}$ generates a dense trajectory at the boundary of the circular table.

By the ergodicity of the Markov shift (see Proposition \ref{markovrandom}), there is an $X_{\theta'} \subset \Sigma_{\theta'}$ of total measure such that the finite sequence $\underline{z}=(1313\dots 13)$ appears in all $\underline{x} \in X_{\theta'}$. Therefore, for every $\underline{x}\in X_{\theta'}$, there is an $m \geq N$ such that $\overline{F}^{(m)}_{\underline{x}}(s,\theta')$ intersects $(a,b) \times (0,\pi)$, that is, the sequence $(s+2\sum_{k=1}^nT_{x_k} \circ T_{x_{k-1}}\circ \dots \circ T_{x_1}(\theta') \mod 2\pi)_n$ intersects $(a,b)$.

Finally, notice that the finite sequence $\underline{z}=(1313\dots 13)$ appears in almost every sequence $\underline{x} \in \Sigma_{\theta}$. Therefore, $(s+2\sum_{k=1}^nT_{x_k} \circ T_{x_{k-1}}\circ \dots \circ T_{x_1}(\theta) \mod 2\pi)_n$ intersects the interval $(a,b)$.

The proof of the second part is very similar to the first one.
\end{proof}

\subsection{Caustics of the Random Billiard Map in the Circle}

In deterministic billiards, a caustic is a curve having the property that if a billiard trajectory is tangent to it at a given point, then it is tangent to it after each collision with the boundary. In the case of a circular deterministic billiard, the caustics are circles of the same center, as we saw in the Section 2.

Remember that given an initial angle $\theta \in (0,\pi)$, for each $\theta' \in \mathcal{C}(\theta)$, we have a circle of the same center, whose radius is $\cos T_i(\theta')$,  that is tangent to the segment of the trajectory that connects the points $(s', \theta')$ and $(s'+2T_i(\theta') \mod 2\pi, T_i(\theta'))$. Remember that $\overline{F}(s',\theta') = (s'+2T_i(\theta') \mod 2\pi, T_i(\theta'))$ with probability $p_i(\theta')$. Therefore, depending on the value of $\alpha$, we can have finite or countably infinite circles which are tangent to some segment of the trajectory of the random billiard map in the circle.  Remember that, given $\theta \in (0,\pi)$, all the possibilities of exit angles of the circular random billiard are found in the set $\mathcal{C}(\theta)$ defined in \ref{def1}. We also remember that given a sequence $\underline{x} \in \Sigma_\theta,$ we can relate it to an admissible sequence $(\theta_n)_n$ defined in \ref{admissivel}.

\begin{definition}
Let $(s,\theta) \in [0,2\pi]\times(0,\pi)$ and $\underline{x}\in \Sigma_{\theta}.$ We define random caustic with respect to trajectory $F_{\underline{x}}(s,\theta)$ in the circle as the circle of radius $r=\inf_{\theta'\in (\theta_n)_n}\cos(\theta')$, where $(\theta_n)_n$ is an admissible sequence with respect to sequence $\underline{x}.$
\end{definition}

Note that random billiards is defined through deterministic billiards. In particular, random circular billiards is defined through deterministic circular billiards. Therefore, that circles whose radius are of the form $\cos(\theta')$, with $\theta' \in \mathcal{C}(\theta)$, are tangent to some billiard trajectory segment.
Moreover, there is always random caustic because in the case where $\inf_{\theta'\in (\theta_n)_n}\cos(\theta')=0$, we will have a single point (circle center) as caustic. In this case we say that the caustic is degenerate, otherwise we say that the random caustic is non-degenerate.

\begin{figure}[!htb]
\centering
\includegraphics[scale=0.045]{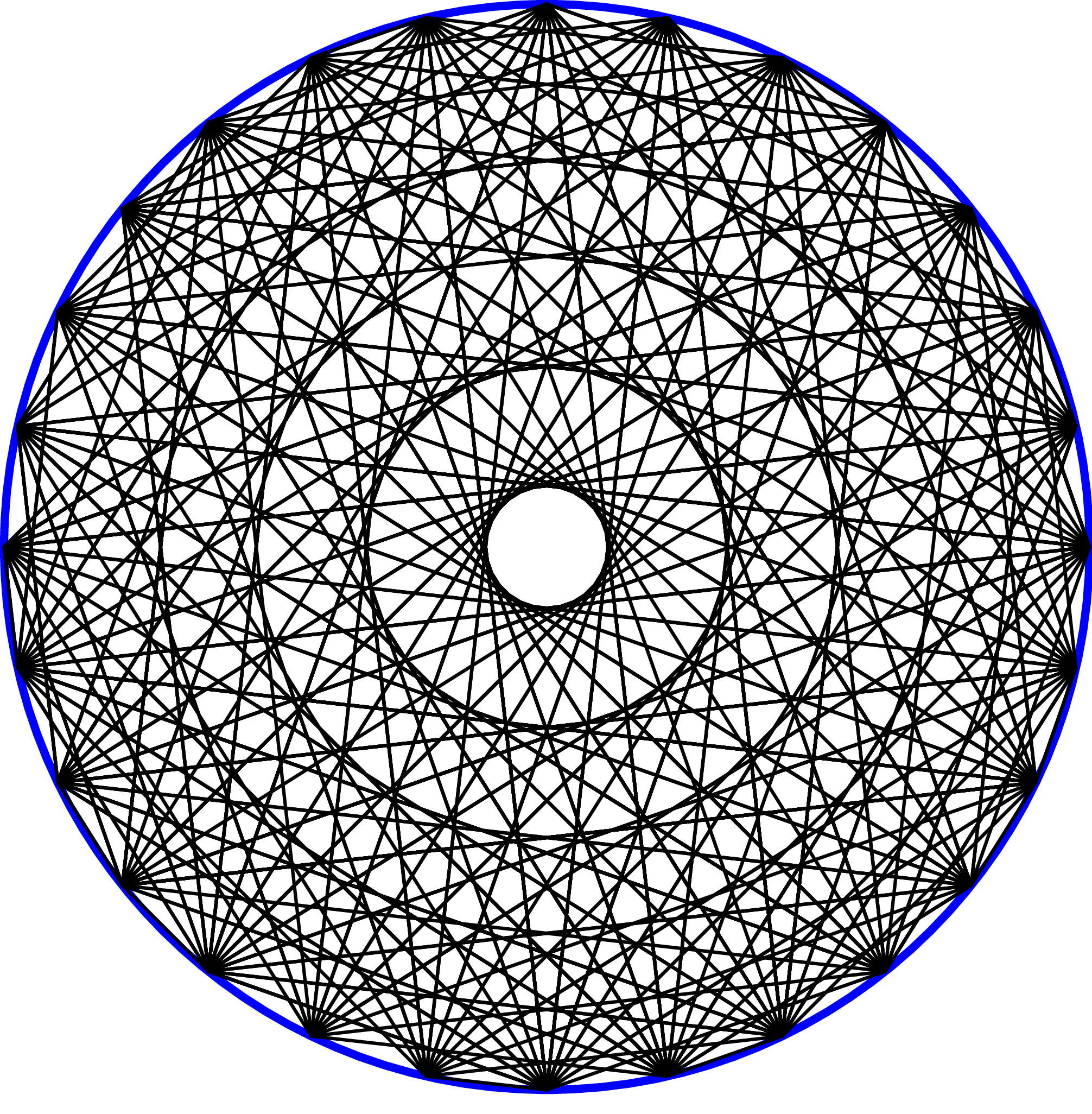} \;\;\;
\includegraphics[scale=0.045]{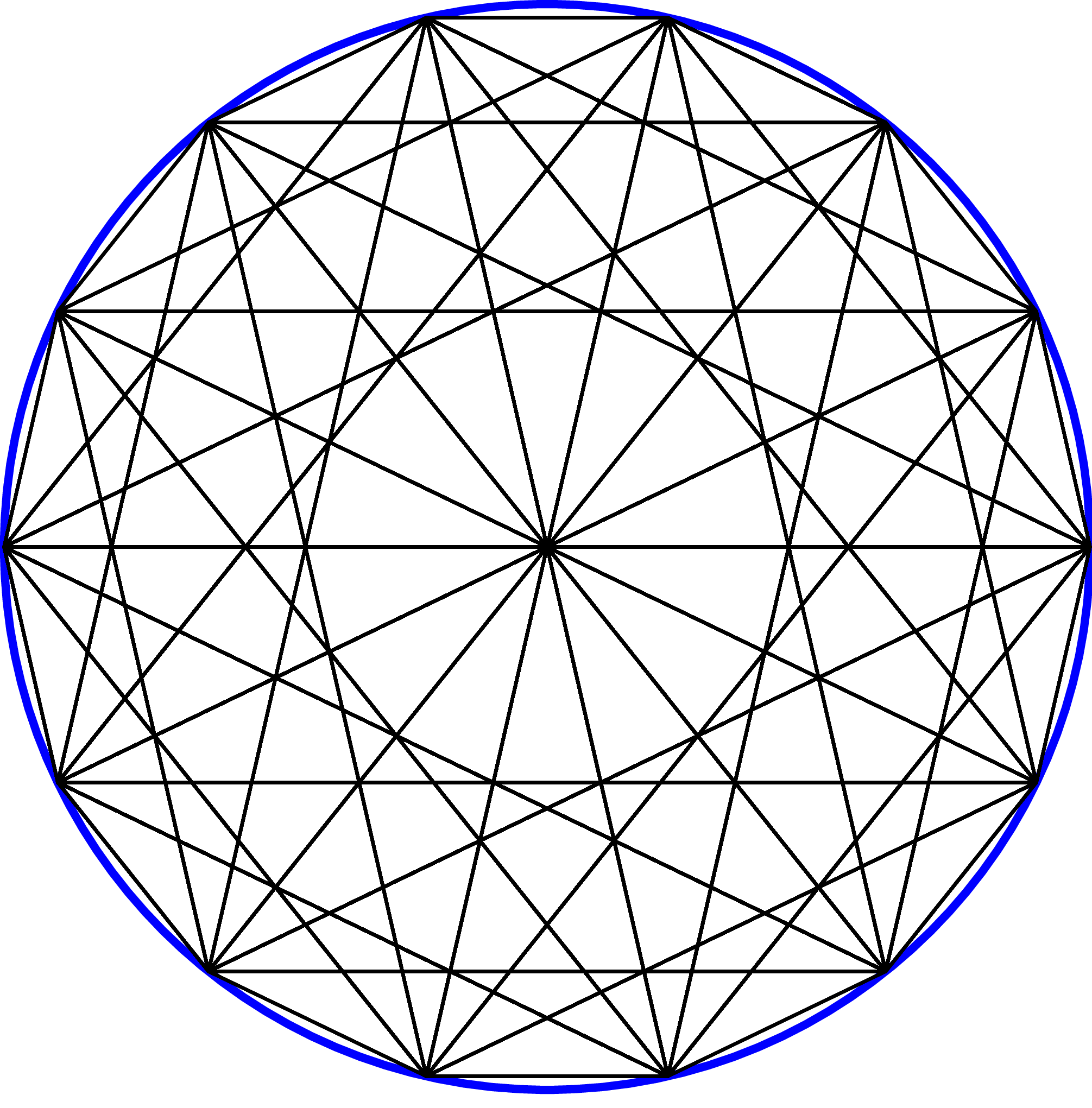}
\caption{In the left, we have the existence of random caustics for the random billiard map with $\alpha=\frac{\pi}{7}$ and initial angle $\theta=\frac{\pi}{20}$. In the right, we have the non-existence of a random caustic for the random billiard map with $\alpha=\frac{\pi}{7}$ and initial angle $\theta=\frac{\pi}{2}$.}\label{fig3}
\end{figure}

Observe that in Figure \ref{fig3} (left) we have $7$ circles, since $\# \mathcal{C}(\frac{\pi}{20})=7$. In addition, we have no trajectories that pass through the center of the circular table and therefore the circle of smallest radius is the caustic of the random billiard map. But for initial angle $\theta = \frac{\pi}{2}$, see Figure \ref {fig3} (right), we have orbits passing through the center of the circular table and, in this case, the center of the circle will be the degenerated caustic. For $\alpha=\frac{\pi}{7}$, if $\theta=\frac{\pi}{2}, \theta=\frac{\pi}{14}, \theta=\frac{3\pi}{14}$ and $\theta=\frac{5\pi}{14}$, we will have degenerated caustic. In the other cases of initial angles $\theta$, we will always have non-degenerated caustics.


In the case that $\frac{\alpha}{\pi}$ is an irrational number, we have that $\mathcal{C} (\theta)$ is countably infinite. So if $\frac{\pi}{2} \not\in \mathcal{C}(\theta)$ and it is not an accumulation point, then there is a neighborhood of $\frac{\pi}{2}$  in $(0,\pi)$ so that no trajectory has no exit angle in that one belonging to that neighborhood. So there is circle $C$ centered on the origin where the trajectory never intersects in the circle $C.$ Then the random billiard map will have a caustic, that is, there will be a circle of smaller radius that is tangent to some part of the path. Otherwise, we will have degenerated caustics.

Let $U$ be the disc of radius 1, let $\partial U$ be the boundary of the disc $U$ and let $\mathcal{A}$ be the circular ring formed by the boundary of the disc $\partial U$ and a random caustic $\gamma_0$.

\begin{theorem}
Consider $(s,\theta) \in [0,2\pi] \times (0,\pi)$ such that $\frac{\pi}{2} \not\in \mathcal{C}(\theta)$ and it is not an accumulation point, and let $\nu$ be the Markov measure.
\begin{enumerate}
\item[1.] If $\frac{\alpha}{\pi}$ is a rational number and $\frac{\theta}{\pi}$ is an irrational number, then, for $\nu$-almost every $\underline{y}\in \big( \mathcal{C}(\theta) \big)^{\mathbb{N}}$, the related trajectory $\overline{F}_{\underline{x}}(s,\theta)$ is dense in $\mathcal{A}$.
\item[2.] If $\frac{\alpha}{\pi}$ is an irrational number and $\theta$ is such that the Markov chain is aperiodic, then, for $\nu$-almost every $\underline{y}\in \big( \mathcal{C}(\theta) \big)^{\mathbb{N}}$, the related trajectory $\overline{F}_{\underline{x}}(s,\theta)$ is dense in $\mathcal{A}$. In particular, if $\frac{\pi}{2}$ is an accumulation point of $\mathcal{C}(\theta)$, then, for $\nu$-almost every $\underline{y}\in \big( \mathcal{C}(\theta) \big)^{\mathbb{N}}$, the related trajectory $\overline{F}_{\underline{x}}(s,\theta)$ is dense in $U$. 
\end{enumerate}
\end{theorem}

\begin{proof}
We will use a similar argument to what was done in Theorem \ref{orbitadensa}, i.e, initially build a dense trajectory in circular ring $\mathcal{A}$ and then conclude, with an ergodic argument, that $\nu$-almost every trajectory is dense in the circular ring $\mathcal{A}.$ Note also that the fact that $\frac{\pi}{2} \not\in \mathcal{C}(\theta)$ and it is not an accumulation point, so every trajectory admits random caustic. That is, there is a circle with a smaller (positive) radius that is tangent to some trajectory segment.

For the first part, take $\theta \in (0,\pi)$ such that $\frac{\theta}{\pi}$ is an irrational number. By hypothesis, assume the existence of a $\theta_0 \in \mathcal{C}(\theta)$ such that $T_i(\theta_0)$ generates a random caustic $\gamma_0$ with $i \in \{1,2,3,4\}.$ That is, the trajectory with exit angle is $T_i(\theta_0)$ is tangent to the circle $\gamma_0.$  Given $\epsilon > 0$ and a point $P$ in the circular ring $\mathcal{A}$, we will show that $\nu$-almost every random billiard trajectory intersects the interior of the ball $B_\epsilon(P)$.


Consider $T_i = T_j^{-1}$ with $j \in \{1,2,3,4\}.$ Since  $p_j(T_i(\theta_0))p_i(\theta_0)>0,$  similarly to the Proposition \ref{densa}, the sequence $\underline{x}= (i,j,i,j,\dots)$ generates a dense orbit at the boundary of the circular table.

Let $r_1, r_2$ be tangent lines to $\gamma_0$ passing through $P$, and let the points $Q_1 \in r_1 \cap \partial U$ and $Q_2 \in r_2 \cap \partial U$ as in Figure \ref{densoanel}. By density of the orbit, we can assume that there is a point $R_1 \in \partial U$ such that $d(R_1, Q_1)<\epsilon$. So the trajectory of $\overline{F}_{\underline{x}}(s_0,\theta_0)$ passing through $R_1$ intersects $B_\epsilon(P)$, since the trajectories passing from $R_1$ is tangent to $\gamma_0$. Due to the ergodicity of the shift, this argument can be made for almost all $\underline{x} \in \Sigma_{\theta_0}$ and therefore for almost every $\underline{x} \in \Sigma_{\theta}$.  

The proof of the second part is very similar to the first one.
\end{proof}

\begin{figure}[htb!]
\centering
\includegraphics[scale=0.3]{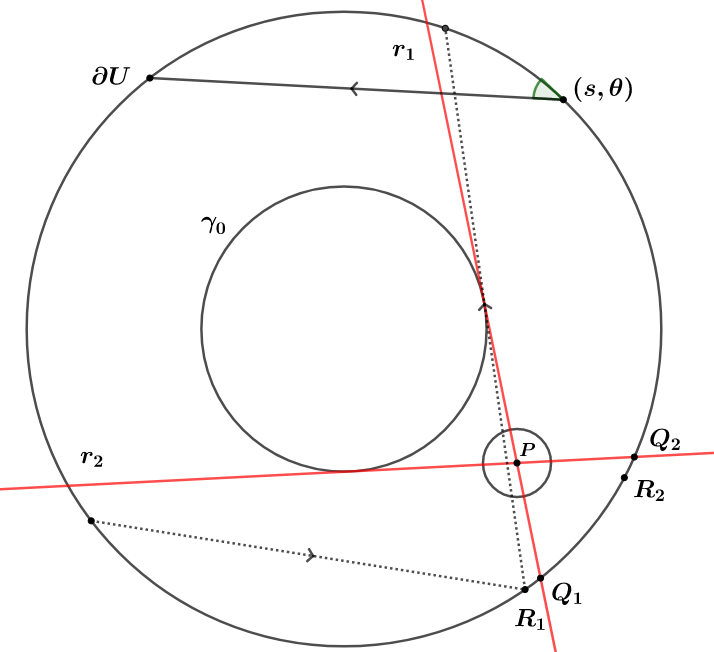}
\caption{Dense trajectory in circular ring $\mathcal{A}.$}\label{densoanel}
\end{figure}

\newpage
\subsection{Lyapunov Exponent}
\begin{definition}
Let $\overline{F}(s,\theta) = F \circ\overline{T}_i(s,\theta)$ be a random billiard map with probability $p_i(\theta)$, and let $\underline{x}\in \Sigma_\theta$. The Lyapunov exponent of $\overline{F}$ with respect to $\underline{x}$ at $(s,\theta)$ in the direction $\vec{v}\in\mathbb{R}^2$ is given by
$$\lambda_{\underline{x}}(\vec{v}, (s,\theta))=\lim_{n \rightarrow \infty} \frac{1}{n}\log\|D_{(s,\theta)}\overline{F}^{(n)}_{\underline{x}}\cdot\vec{v}\|,$$ when the limit exists and when $D_{(s,\theta)}\overline{F}^{(n)}_{\underline{x}}$ is well defined where $\overline{F}^{(n)}_{\underline{x}} = \overline{F}_{x_n} \circ \overline{F}_{x_{n-1}} \circ \dots \circ \overline{F}_{x_2} \circ \overline{F}_{x_1}.$  
\end{definition}

The Lyapunov exponent of the Feres random map $T,$ at point $\theta$ with respect to $\underline{x}\in \Sigma_\theta$ is given by $$\lambda_{\underline{x}}(\theta)= \lim_{n \rightarrow \infty} \frac{1}{n} \log |(T^{(n)}_{\underline{x}})' (\theta)|$$ when the limit exists. Note that the derivative of maps $T_i$ are $\pm1,$ and therefore $$\lambda_{\underline{x}}(\theta)=\lim_{n \rightarrow \infty} \frac{1}{n} \log |(T^{(n)}_{\underline{x}}(\theta))'|=0$$ for all $\underline{x} \in \Sigma_\theta.$

\begin{proposition}
The Lyapunov exponent of the random billiard map in the circle is zero for all $(s,\theta) \in [0,2\pi)\times(0,\pi)$, $\underline{x}\in \Sigma_\theta$ and every direction $\vec{v} \in \mathbb{R}^2$.
\end{proposition}

\begin{proof}
First, observe that the derivative $D_{(s,\theta)}\overline{F}^{(n)}_{\underline{x}}$ of the random billiard map in the circle is 
$$D_{(s,\theta)}\overline{F}^{(n)}_{\underline{x}}=\begin{bmatrix}
1 & A_n \\
0 & B_n
\end{bmatrix},$$
where
$$A_n= \sum_{k=1}^n 2T_{x_k}'(T_{x_{k-1}}\circ \dots \circ T_{x_1}(\theta))T'_{x_{k-1}}(T_{x_{k-2}}\circ \dots \circ T_{x_1}(\theta))\dots T'_{x_2}(T_{x_1}(\theta))T'_{x_1}(\theta)$$
and
$$B_n=T_{x_k}'(T_{x_{k-1}}\circ \dots \circ T_{x_1}(\theta))T'_{x_{k-1}}(T_{x_{k-2}}\circ \dots \circ T_{x_1}(\theta))\dots T'_{x_2}(T_{x_1(\theta)})T'_{x_1}(\theta).$$

By linearity of the maps $T_{x_i}$, we conclude that $A_n \in \{-2n,\dots, 0, \dots,2n\}$ and $B_n \in \{-1,1\}$. Hence, taking any direction $\vec{v}=(v_1,v_2)\in \mathbb{R}^2$, we obtain that 
$$\frac{1}{n}\log \|D_{(s,\theta)} \overline{F}_{\underline{x}}^{(n)} \cdot \vec{v} \| =\frac{1}{n} \log \| (v_1+v_2A_n, v_2B_n) \|.$$
Since $A_n$ and $B_n$ are bounded, it follows that
$$\lambda_{\underline{x}}(\vec{v},(s,\theta))=\lim_{n\to \infty} \frac{1}{n} \log \|D_{(s,\theta)}\overline{F}_{\underline{x}}^{(n)} \cdot \vec{v}\| = 0 $$
for all $(s,\theta)\in [0,2\pi)\times(0,\pi)$ and every direction $\vec{v}\in \mathbb{R}^2$. 
\end{proof}




\section*{Acknowledgement(s)}
The authors thank Pablo D. Carrasco and Sylvie O. Kamphorst for their suggestions. Túlio Vales was supported by Coordena\c c\~ao de Aperfei\c coamento de Pessoal de N\'ivel Superior (CAPES-Brasil) and Funda\c c\~ao de Amparo \`a Pesquisa do Estado de Minas Gerais (FAPEMIG-Brasil).

Finally, we thank for the reviewer for the suggestions that improved our manuscript.

\bibliography{bibliografia}
\bibliographystyle{tfnlm}

\appendix
\section{}


The deterministic billiard in the infinite horizon pipeline is defined on a strip bounded by two parallel lines. The deterministic billiard map in the infinite pipeline is then $F(s,\theta) = (s_1(s,\theta), \theta)$, and the random billiard map in the infinite pipeline is $\overline{F}= F\circ \overline{T}$. Its derivative at $(s_0, \theta_0)$ is 
$$D\overline{F}^{(1)}_{\underline{x}}(s_0,\theta_0) = \begin{bmatrix}
-1 & \pm \frac{l_{01}}{\sin T_{x_1}(\theta_0)} \\
0 & -1 
\end{bmatrix}.$$

\begin{proposition}\label{6.2}
If $\frac{\alpha}{\pi}$ is a rational number, then the Lyapunov exponent of the random billiard map in the infinite pipeline is zero for all $(s,\theta) \in [0,2\pi)\times(0,\pi)$, $\underline{x}\in \Sigma_\theta$ and every direction $\vec{v} \in \mathbb{R}^2$.
\end{proposition}

\begin{proof}
Inductively, we obtain that
$$D\overline{F}^{(n)}_{\underline{x}}(s_0,\theta_0) = (-1)^n\begin{bmatrix}
-1 & \pm \frac{l_{01}}{\sin T_{x_1}(\theta_0)} \pm \frac{l_{12}}{\sin T_{x_2} \circ T_{x_1}(\theta_0)} \pm \dots \pm \frac{l_{(n-1) n}}{\sin T_{x_n}\circ \dots \circ T_{x_1}(\theta_0)} \\
0 & -1 
\end{bmatrix}.$$

Hence, taking any direction $\vec{v}=(v_1,v_2) \in \mathbb{R}^2$, it follows that
\begin{multline*}
\Big\| D_{(s_0,\theta_0)}\overline{F}_{\underline{x}}^{(n)}\cdot v\Big\| =\\
\shoveleft{\qquad = \Big\| \Big( -v_1 + v_2 \Big(\pm \frac{l_{01}}{sin T_{i_1}(\theta_0)} \pm \frac{l_{12}}{sin T_{i_2} \circ T_{i_1}(\theta_0)} \pm \dots}\\
\shoveright{\dots \pm \frac{l_{(n-1) n}}{sin T_{i_n}\circ \dots \circ T_{i_1}(\theta_0)} \Big), -v_2 \Big)\Big\|}\\
\shoveleft{\qquad \leq \Big\| (v_1,v_2) \Big\| + \Bigg\| \Bigg(v_2 \Big(\pm \frac{l_{01}}{sin T_{i_1}(\theta_0)} \pm \frac{l_{12}}{sin T_{i_2} \circ T_{i_1}(\theta_0)} \pm \dots}\\
\dots \pm \frac{l_{(n-1) n}}{sin T_{i_n}\circ \dots \circ T_{i_1}(\theta_0)} \Big), 0\Bigg) \Bigg\|.
\end{multline*}

Since $\frac{\alpha}{\pi}$ is a rational number, we have that $\mathcal{C}(\theta_0) =\{\theta_1, \theta_2, \dots, \theta_m \}$. Let $L$ be the maximum between the numbers $l_{01}, l_{12}, \dots, l_{(n-1)n}$. Take an angle $\theta \in \mathcal{C}(\theta_0)$ such that $\frac{1}{sin \theta} \geq \frac{1}{sin \theta_i}$ for all $i \in \{1,2,\dots,n\}$. Then,
\begin{align*}
\Big\| D_{(s_0,\theta_0)}\overline{F}_{\underline{x}}^{(n)}\cdot v\Big\| &\leq \|(v_1,v_2)\| +  \sqrt{v_2^2 \Big(\frac{l_{01}}{sin \theta_1} + \frac{l_{12}}{sin \theta_2 } + \dots + \frac{l_{(n-1)n}}{sin \theta_n}\Big)^2}  \displaybreak[0]\\
&= \|(v_1,v_2)\| +  v_2\Big(\frac{l_{01}}{sin \theta_1} + \frac{l_{12}}{sin \theta_2 } + \dots + \frac{l_{(n-1)n}}{sin \theta_n}\Big)
\\
&\leq \|(v_1,v_2)\| + L \Big(\frac{1}{sin \theta_1} + \frac{1}{sin \theta_2 } + \dots + \frac{1}{sin \theta_n}\Big) \\
&\leq \|(v_1,v_2)\| + \frac{Ln}{sin \theta}.
\end{align*}
Therefore,
$$\lambda_{\underline{x}}(\vec{v},(s_0,\theta_0)) = \lim_{n\rightarrow\infty} \frac{1}{n}\log \|D_{(s_0,\theta_0)}\overline{F}^{(n)}_{\underline{x}}\cdot\vec{v}\| = \lim_{n\rightarrow\infty}\frac{1}{n}\log \Bigg(\|(v_1,v_2)\| + \frac{Ln}{sin \theta}\Bigg)=0.$$
\end{proof}



\end{document}